\documentclass[12pt]{amsart}
\usepackage{amsmath,amssymb,amsthm,graphicx} 
\usepackage[T2A]{fontenc}
\usepackage[english]{babel}
\usepackage[left=3cm,right=3cm,top=2.5cm,bottom=2.5cm]{geometry}
\usepackage[parfill]{parskip} 
\usepackage{marginnote} 
\usepackage{easyReview}
\usepackage{hyperref}
\usepackage{xcolor}
\usepackage{setspace}
\usepackage{fancyhdr}
\theoremstyle{indented}
\newtheorem{theorem}{Theorem}

\newtheorem{lemma}{Lemma}

\theoremstyle{definition} 
\newtheorem{defn}{Definition}

\theoremstyle{remark} 
\newtheorem*{remark}{Remark}

\DeclareMathOperator{\NN}{\mathbb{N}}
\DeclareMathOperator{\RR}{\mathbb{R}}
\DeclareMathOperator{\TT}{\mathbb{T}}

\DeclareMathOperator{\ra}{\rightarrow}
\DeclareMathOperator*{\esssup}{ess\,sup}
\DeclareMathOperator*{\essinf}{ess\,inf}
\newcommand{\lfrf}{\lfloor m^{\alpha} \rfloor}
\newcommand{\lcrc}{\lceil m^{\alpha} \rceil}

\usepackage{fancyhdr}
\title{A version of Marstrand's theorem on a discrete metric space}
\author{Leonid Gorbunov}
\address{Department of Mathematics and Computer Science, St. Petersburg State University; 199178, 14th line 29, Vasilyevsky Island, St. Petersburg, Russia}
\email{leon.tyumen@icloud.com}
\thanks{Supported by the Russian Science Foundation grant 24-71-10011.}

\begin{document}  

\begin{abstract}
We present and prove the version of Marstrand's theorem for discrete metric space. We provide explicit estimates of the quotient of upper and lower densities of measures on this space.
\end{abstract}

\clearpage\maketitle
\thispagestyle{empty}

\section{introduction}
The study of local behavior of measures is an important theme in geometric measure theory. Let $(X, d)$ be a metric space and let $\mu$ be a Radon measure on $X$. We denote the closed ball with center $x$ and radius $r$ by $B(x, r)$
\begin{defn}
Define the upper and lower $\alpha$-densities of $\mu$ at $x \in X$:
        \begin{equation}\label{densities}
            \begin{aligned}
            \Theta^{\alpha, *}(\mu, x) = \limsup_{r \ra 0_+} \frac{\mu \left ( B_r(x)\right )}{r^{\alpha}}; \\ 
            \Theta_{\alpha, *}(\mu, x) = \liminf_{r \ra 0_+} \frac{\mu \left ( B_r(x)\right )}{r^{\alpha}}.
            \end{aligned}
        \end{equation}
    If they agree, their value is called the $\alpha$-\textit{density} of $\mu$ at $x$ and is denoted by $\Theta^{\alpha}(\mu, x)$.    
\end{defn}
\begin{theorem}[Marstrand~\cite{Marstrand}, 1964]
Let $\mu$ be a Radon measure on $\RR^d$, let $A$ be a set of positive measure, and let $\alpha \in [0, d]$ be a real number. Suppose that for $\mu$-a.e. $x \in A$ the $\alpha$-density of $\mu$ at $x$ exists, is finite and positive. Then $\alpha$ is integer.
\end{theorem}

Moreover, there is a strengthening of Marstrand's theorem proved by Preiss in~\cite{Preiss}.
\begin{theorem}[Preiss, 1987]\label{common}
 Let $\alpha$ be a real positive non-integer number. There is $\varepsilon = \varepsilon(\alpha, d)$ such that 
 \begin{equation*}
     \mu \left(\left\{x \, \Big  | \,  0 < \Theta^{\alpha, *}(\mu, x) < (1+\varepsilon) \Theta_{\alpha, *}(\mu, x) < \infty \right\} \right) = 0
 \end{equation*}
 for every Borel measure $\mu$ on $\mathbb{R}^d$.
\end{theorem}

The proofs of these theorems heavily use properties of the space $\RR^d$ as their important part. Unfortunately, there is almost no information on the behavior of $\varepsilon(\alpha, d)$. The open problem of its estimation is one of many in regarding densities defined above. Other unsolved problems could be found in the paper~\cite{Lellis} of C. de Lellis. More information about densities of measures could be found either in the mentioned papers~\cite{Lellis} and ~\cite{Preiss} or in the book~\cite{Mattila} of P. Mattila.

The theorems proved in this paper are quite similar to that of Marstrand and Preiss. However, we will be working with the space of infinite sequences $\{0, 1, \dots, m-1\}^{\NN}$, instead of $\RR^d$.

I wish to thank Rami Ayoush for attracting my attention to this problem, to Pertti Mattila for consultation and to Dmitriy Stolyarov for academic advising, reading the paper and improving the presentation.

\section{Preliminaries}

Given an integer $m \geqslant 2$, let us define the set $\TT$ of infinite sequences of integers from the interval $[0, m-1]$ as follows: 

\begin{equation*}
    \TT = \{0, 1, \dots , m-1\}^{\NN}.
\end{equation*}

To simplify notation now and hereafter, we will denote sequences using only one symbol, and call them \textit{points} or \textit{numbers}. For example, $a=(a_1, a_2 \dots)$. Define functions $d, n \colon \TT \times \TT \ra \RR$ as follows: $n(a, b)$ is the first difference position of the sequences $a=(a_1,a_2, \dots )$, $b=(b_1,b_2, \dots)$, and $d(a, b) = m^{-n(a, b)+1}$. Note that $(\TT, d)$ is a metric space.

\begin{defn}
 Let $n$ be a positive integer. Denote $B_n(x):= B(x, m^{-n})$. The balls $B_n(x)$ are called \textit{intervals} or \textit{$m$-adic intervals} of level $n$. In particular, the space $\TT$ itself is the interval of level $0$.
\end{defn}

\begin{remark} \label{seq}
    It is clear that $B_n(x)$ is the set of sequences that match $x$ in the first $n$ positions.
\end{remark}

\begin{remark} 
    In the case of $\TT$ the definition~\eqref{densities} of $\alpha$-densities may be restated:
        \begin{equation}
            \begin{aligned}
            \Theta^{\alpha, *}(\mu, x) = \limsup_{n \in \NN, \, n \ra \infty} m^{\alpha n} \mu \left ( B_n(x)\right ); \\ 
            \Theta_{\alpha, *}(\mu, x) = \liminf_{n \in \NN, \, n \ra \infty}  m^{\alpha n} \mu \left ( B_n(x)\right ).
            \end{aligned}
        \end{equation}
\end{remark}

Now and henceforth all measures are assumed to be nonzero, finite and defined on the Borel $\sigma$-field of $\TT$.

\section{Main results}
 Assume that the upper $\alpha$-densities are finite and the lower $\alpha$-densities are positive $\mu$-a.e. Then we can define the values
    \begin{equation*}\label{cloc}
        \mathcal{C}_{loc}(\mu, \alpha) = \esssup_{x \in \TT} \frac{\Theta^{\alpha, *}(\mu, x)}{\Theta_{\alpha, *}(\mu, x)},
    \end{equation*}
and 
    \begin{equation*}
        \mathcal{C}(\mu, \alpha) = \frac{\esssup_{x \in \TT} \Theta^{\alpha, *}(\mu, x)}{\essinf_{x \in \TT}\Theta_{\alpha, *}(\mu, x)},
    \end{equation*}
if the numerator of the fraction is finite and the denominator of the fraction is positive. It follows from the definition that
    \begin{equation}\label{basic_lower}
        \mathcal{C}(\mu, \alpha) \geqslant  \mathcal{C}_{loc}(\mu, \alpha) \geqslant 1.
    \end{equation} 
What can we say about these values? The author does not know the sharp estimate. However, Theorems~\ref{mainlower} and~\ref{mainupper} provide a partial answer.
We denote the largest integer that is less than or equal to a real number $x$ by $\lfloor x \rfloor$, and the smallest integer greater than or equal to $x$ by $\lceil x \rceil$.
\begin{theorem}\label{mainlower}
    For any integer $m \geqslant 2$, any real $\alpha \in (0, \, 1)$ and any Radon measure $\mu$ on $\TT$ the following inequality holds:
    \begin{equation}
        \mathcal{C}_{loc}(\mu, \alpha) \geqslant \max_{i, j \in \NN} \left( \frac{m^{\alpha i}}{\lfloor m^{\alpha i} \rfloor}, \frac{\lceil m^{\alpha j} \rceil}{m^{\alpha j}}\right).
    \end{equation}
\end{theorem}
If $m^{\alpha}$ is an integer, Theorem~\ref{mainlower} gives the trivial estimate $\mathcal{C}_{loc}(\mu, \alpha) \geqslant 1$ only, which is already known from~\eqref{basic_lower}. So, we will prove Theorem~\ref{mainlower} under the assumption $m^{\alpha} \notin \NN$.

\begin{theorem}\label{mainupper}
    There is a measure $\mu$ such that
    \begin{equation}\label{mainupper_eq}
        \mathcal{C}(\mu, \alpha) \leqslant \frac{\lceil m^{\alpha} \rceil}{\lfloor m^{\alpha} \rfloor}.
    \end{equation}
\end{theorem}
In particular, these estimates do not depend on $m$ but only on $m^{\alpha}$. Informally, we can say that the estimates are dimension-free.

\section{Proof of Theorem~\ref{mainlower}}

\begin{defn}\label{defn1}
 Let $\mu$ be a measure. Define the sequence of functions $\{f_n\}_{n=0}^{\infty} $ by the rule
    \begin{equation*}
f_n(x) = m^{\alpha n} \mu(B_n(x)).
    \end{equation*}
\end{defn}
It follows from the definition that the functions $f_n$ are locally constant. Consequentely, they are continious on the space $\TT$ and measurable with respect to the measure $\mu$.

Since we can pass from $m$-adic intervals to $m^d$-adic, it is sufficient to prove the inequality
\begin{equation}
    \mathcal{C}_{loc}(\mu, \alpha) \geqslant \max \left( \frac{m^{\alpha }}{\lfloor m^{\alpha } \rfloor}, \frac{\lceil m^{\alpha } \rceil}{m^{\alpha }}\right).
\end{equation}

\begin{lemma}\label{dirichlet}
    Let $w, u > 1$ be real numbers such that $w$ is not integer. Then for every $\delta > 0$ there exists $\tau>0$ such that for any integer $n \in [1, u)$ from any positive real numbers $z_1, \dots, z_n$ with sum $w$ there can be chosen one that lies in the set $\left [ \tau, \frac{w}{\lceil w \rceil} \right ] \cup \left[\frac{w}{\lfloor w \rfloor}-\delta, w\right)$.
\end{lemma}
\begin{proof}
    Assume all the numbers $z_1, \dots, z_{n}$ lie in the set $\left(0, \tau \right) \cup \left(\frac{w}{\lceil w \rceil}, \frac{w}{\lfloor w \rfloor}-\delta\right)$. Let $J$ be the set of indices $j$ for which $z_j \in \left(\frac{w}{\lceil w \rceil}, \frac{w}{\lfloor w \rfloor}-\delta\right)$. The inequality $w \geqslant \sum_{j \in J} z_j > |J| \cdot \frac{w}{\lceil w \rceil}$ imply $|J| < \lceil w \rceil$. Then $|J| \leqslant \lfloor w \rfloor$ and $\sum_{j \in J} z_j \leqslant w-\lfloor w \rfloor \delta$. So, $\sum_{j \notin J} z_j \geqslant \lfloor w \rfloor \delta $. According to the pigeonhole principle, there is a number among them that is not less than $\frac{\lfloor w \rfloor \delta}{n}$. Consequentely, $\tau \geqslant \frac{\lfloor w \rfloor \delta}{n} \geqslant \frac{\lfloor w \rfloor \delta}{u}$.
\end{proof}

\begin{lemma}\label{prop_of_big_intervals}
    Let $\mu$ be a measure, let $\Gamma_1, \Gamma_2,  \dots$ be a sequence of sets such that the set $\Gamma_n$ is the union of some intervals of level $n$ --- one calls them "\textit{marked intervals}". Suppose that there exists a positive number $K$ such that for any positive integer $n$ and any interval $I_k^n$ of level $n$ the following inequality holds: 
    \begin{equation}\label{lemma}
        \mu(I^n_k \cap \Gamma_{n+1}) \geqslant K \cdot \mu(I^n_k).
    \end{equation} 
    Then for $\mu$-a.e. point $x \in \TT$ the sequence of the balls $\{B_n(x)\}_{n=0}^{\infty}$ contains arbitrarily long sequences of consecutive marked balls.
\end{lemma}
\begin{proof}
    It is sufficient to prove that for any positive integer $d$ the set of points $x \in \TT$ such that the sequence of the balls $\{B_n(x)\}_{n=0}^{\infty}$ does not contain a sequence of $d$ consecutive marked balls, has zero $\mu$-measure.
    
     It follows from the statement of the lemma, that for each interval $I_k^n$ the measure $\mu$ of the set of the points $x\in I^n_k$ such that the balls $B_{n+1}(x), \dots , B_{n+d}(x)$ are marked, is not less than $K^d  \cdot \mu(I^n_k)$. So $m$-adic intervals may be replaced by $m^d$-adic intervals and we can consider only the case $d=1$.

    One may see that the set under consideration is precisely equal to $\TT \backslash \left ( \bigcup_{n=1}^{\infty} \Gamma_n \right )$. For every positive integer $n'$ the set $\TT \backslash \left ( \bigcup_{n=1}^{n'-1} \Gamma_n \right )$ is a union of some intervals of level $n'$, and therefore after direct application of~\eqref{lemma} we get that
    \begin{equation}
        \mu \left ( \TT \backslash \left ( \bigcup_{n=1}^{n'} \Gamma_n \right ) \right ) \leqslant (1-K) \cdot \mu \left (\TT \backslash \left ( \bigcup_{n=1}^{n'-1} \Gamma_n \right ) \right ).
    \end{equation}
    Consequently, 
    \begin{equation}
        \mu \left ( \TT \backslash \left ( \bigcup_{n=1}^{n'} \Gamma_n \right ) \right ) \leqslant (1-K)^{n'} \mu(\TT).
    \end{equation} 
    Choosing sufficiently large $n'$, we obtain the required statement.
\end{proof}
\begin{proof}[Proof of Theorem~\ref{mainlower}]
    At first, we fix a real number $\delta>0$ such that  $\frac{m^{\alpha}}{\lfrf} - \delta > 1$.
Denote some interval of level $n$ that intersects the support of $\mu$ by $I^n_k$. Consider all its subintervals of level $n+1$ that also intersect the support of $\mu$, denote them by $I^{n+1}_{m(k-1)+{i_j}}$. Points $x_{i_j}$ are chosen from the intersection of intervals $I^{n+1}_{m(k-1)+{i_j}}$ respectively and the support of $\mu$.
By the definition of the $f_n$,
\begin{equation}
    m^{\alpha} = \frac{f_{n+1}(x_{i_1})}{f_n(x_{i_1})}+\dots+\frac{f_{n+1}(x_{i_s})}{f_n(x_{i_s})}.
\end{equation}
After application of Lemma~\ref{dirichlet} for constants $w=m^{\alpha}$, $u=m$, $z_j = \frac{f_{n+1}(x_{i_j})}{f_n(x_{i_j})}$ and fixed in the beginning of the proof $\delta$, one gets that for some $j \in \{1, \dots, s\}$ the following statement holds:
\begin{equation}\label{basic}
    \frac{f_{n+1}(x_{i_j})}{f_n(x_{i_j})} \in \left [ \tau, \frac{m^{\alpha}}{\lcrc} \right ] \cup \left[ \frac{m^{\alpha}}{\lfrf}-\delta, m^{\alpha}\right).
\end{equation}
We call the subinterval $I^{n+1}_{m(k-1)+i_j}$ a \textit{little marked interval} if $\frac{f_{n+1}(x_{i_j})}{f_n(x_{i_j})} \in \left [ \tau, \frac{m^{\alpha}}{\lcrc} \right ]$. Similarly, if $\frac{f_{n+1}(x_{i_j})}{f_n(x_{i_j})} \in \left[ \frac{m^{\alpha}}{\lfrf}-\delta, m^{\alpha}\right)$, the subinterval is called a \textit{big marked interval}.

Let us apply Lemma~\ref{prop_of_big_intervals} for the number $K=\tau m^{-\alpha}$ and conclude that for $\mu$-a.e. points $x \in \TT$ in the sequence $\left \{ \frac{f_{n+1}(x)}{f_n(x)}\right \}_{n=0}^{\infty}$ there are arbitrarily long sequences of consecutive marked intervals. If for sufficiently large $n$ there are only either big marked intervals or little marked intervals in the sequence $\{B_n(x)\}_{n=0}^{\infty}$ then for any positive integer $d$ either the estimate $\mathcal{C}_{loc}(\mu, \alpha) \geqslant \left(\frac{m^{\alpha}}{\lfrf} - \delta  \right)^d$, or the estimate $\mathcal{C}_{loc}(\mu, \alpha) \geqslant \left(\frac{\lcrc}{m^{\alpha}} \right)^d$ holds, which is impossible if $d$ is sufficiently large. 
Hence, we conclude that there are infinitely many both big intervals and little intervals, and consequently 
\begin{equation}
    \mathcal{C}_{loc}(\mu, \alpha) \geqslant \max \left(\frac{m^{\alpha}}{\lfrf} - \delta, \frac{\lcrc}{m^{\alpha}}  \right).
\end{equation}
The choice of arbitrarily small $\delta$ finished the proof of Theorem~\ref{mainlower}.
\end{proof}

\section{Proof of Theorem~\ref{mainupper}}
Recall that a measure $\mu$ is called \textit{uniform} if for arbitrary points $x, y$ that belong to the support of $\mu$ and an arbitrary real number $r>0$ the equality $\mu \left(B(x, r)\right) = \mu \left(B(y, r)\right)$ holds. Moreover, if there are positive constants $c$, $n$ such that $\mu \left(B(x, r)\right) = c r^n$, then a measure $\mu$ is called \textit{$n$-uniform}. There are many open problems in the theory of uniform measures in $\RR^d$. Marstrand proved in~\cite{Marstrand} that there is no $\alpha$-uniform measure for non-integer $\alpha$. Preiss proved in~\cite{Preiss} that for $n=1$ and $n=2$ every $n$-uniform measure $\mu$ is flat, i.e. there are a positive constant $c$ and an $n$-plane $L$ such that $\mu = c \mathcal{H}^n \llcorner L$. He also showed that for the cone $C = \{(x_1, x_2, x_3, x_4) \in \RR^4 \mid x_4^2=x_1^2+x_2^2+x_3^2\}$ the measure $\mathcal{H}^3 \llcorner C$ is $3$-uniform. Moreover, Kowalski and Preiss proved in~\cite{KoP} that every $(d-1)$-uniform measure in $\RR^{d}$ is either flat measure or, up to isometry, the measure $c\mathcal{H}^{d-1} \llcorner (C \times \RR^{d-4})$. In the case of $2 < n < d-1$ there was no classification of $n$-uniform measures, until A. Dali Nimer in~\cite{Nimer} classified $3$-uniform measures in $\RR^5$ and constructed new examples of $3$-uniform measures in arbitary codimension.

In the case of $\TT$ all uniform measures can be described explicitly.
\begin{lemma}\label{uniform}
    Let $\mu$ be a uniform measure and let $I$ be an arbitary interval of level $n$ that intersects the support of $\mu$. Denote by $s_I$ the number of subintervals of $I$ of level $n+1$ that intersect the support of $\mu$. Then, $s_I$ depends only on $n$, not on the choice of $I$. We will denote this value by $s_n$. 
\end{lemma}
\begin{proof}
    Denote the measure of any interval of level $n$ intersecting the support of $\mu$ by $u_n$. Now count a measure of an arbitrary interval of level $n$ that intersects the support of $\mu$: $u_n = s_n u_{n+1}$. Consequentely, $s_n$ depends only on the level of the interval.
\end{proof}

Lemma~\ref{uniform} establishes a bijection between uniform measures (up to shuffling the intervals) and pairs consisting of a positive real number $x_0$ and a sequence of positive integers $(s_0, s_1, s_2, \dots)$. Indeed, if there is a uniform measure $\mu$, set $x_0 = \mu(\mathbb{T})$, and the sequence $(s_0, s_1, s_2, \dots)$ appears from the previous lemma.

Assume that the number $x_0$ and the sequence $(s_0, s_1, s_2, \dots)$ are fixed. Let us describe the measure $\mu$ required for the proof of Theorem~\ref{mainupper}. The support of $\mu$ is the set of sequences $(a_1, a_2, \dots)$ such that $a_i \in \{0, 1, ..., s_i - 1\}$. Define the measure of an interval of level $n$ by the following rule: its measure is equal to $\frac{x_0}{s_0  \dots  s_{n-1}}$, if the interval intersects the support of $\mu$, and equals to $0$, if it does not intersect.

\begin{proof}[Proof of Theorem~\ref{mainupper}]
    In the example measure each $s_i$ is either equal to $\lfrf$ or to $\lcrc$. Then the transition from $f_n$ to $f_{n+1}$ is either multiplication by $\frac{m^{\alpha}}{\lfrf}$ or division by $\frac{\lcrc}{m^{\alpha}}$. Define $s_n$ inductively by the following rule:
    \begin{equation}
        s_{n} = \begin{cases}
            \lcrc,  & \mbox{if } f_{n-1} \frac{m^{\alpha}}{\lfrf} > \frac{\lcrc}{\lfrf}x_0;\\
            \lfrf, & \mbox{if } f_{n-1} \frac{m^{\alpha}}{\lfrf} \leqslant \frac{\lcrc}{\lfrf}x_0.
          \end{cases}
    \end{equation}
     Note that the first case is chosen if and only if $f_{n-1} > \frac{\lcrc}{m^{\alpha}}x_0$, and therefore after one step $f_n = f_{n-1} \cdot \frac{m^{\alpha}}{\lcrc} > x_0$. Consequently,
    \begin{equation}
        \begin{aligned}
            \Theta^{\alpha, *}(\mu, x) = \limsup_{n \ra \infty} f_n(x) \leqslant \frac{\lcrc}{\lfrf}x_0; \\ 
            \Theta_{\alpha, *}(\mu, x) = \liminf_{n \ra \infty} f_n(x) \geqslant x_0,
        \end{aligned}
    \end{equation}
    and the inequality~\eqref{mainupper_eq} is true.
\end{proof}

\end{document}